\documentclass[12pt, reqno]{amsart}
\usepackage{amsmath, amsthm, amscd, amsfonts, amssymb, graphicx, color}
\usepackage[bookmarksnumbered, colorlinks, plainpages]{hyperref}

\newtheorem{theorem}{Theorem}[section]

\newtheorem{proposition}[theorem]{Proposition}
\newtheorem{corollary}[theorem]{Corollary}
\theoremstyle{definition}
\newtheorem{definition}[theorem]{Definition}
\newtheorem{example}[theorem]{Example}

\theoremstyle{remark}
\newtheorem{remark}[theorem]{Remark}
\numberwithin{equation}{section}

\begin{document}
\setcounter{page}{1}

\title[Functional Version]{Functional Version for Furuta Parametric Relative Operator Entropy}

\author[M. Ra\"{\i}ssouli, S. Furuichi]{Mustapha Ra\"{\i}ssouli$^{1,2}$ and Shigeru Furuichi$^{3}$}

\address{$^{1}$ Department of Mathematics, Science Faculty, Taibah University,
Al Madinah Al Munawwarah, P.O.Box 30097, Zip Code 41477, Saudi Arabia.}

\address{$^{2}$ Department of Mathematics, Science Faculty, Moulay Ismail University, Meknes, Morocco.}

\address{$^3$Department of Information Science, College of Humanities and Sciences, Nihon University, 3-25-40, Sakurajyousui, Setagaya-ku, Tokyo, 156-8550, Japan.}

\email{\textcolor[rgb]{0.00,0.00,0.84}{raissouli.mustapha@gmail.com}}
\email{\textcolor[rgb]{0.00,0.00,0.84}{furuich@chs.nihon-u.ac.jp}}



\date{Received: xxxxxx; Revised: yyyyyy; Accepted: zzzzzz.}

\begin{abstract}
Functional version for the so-called Furuta parametric relative operator entropy is here investigated. Some related functional inequalities are also discussed. The theoretical results obtained by our functional approach immediately imply those of operator versions in a simple, fast and nice way.
\end{abstract}

\maketitle

\section{\bf Introduction}

Let $H$ be a complex Hilbert space. We denote by ${\mathcal B}(H)$ the $\mathbb{C}^*$-algebra of bounded linear operators acting on $H$ and by ${\mathcal B}^{+*}(H)$ the open cone of all (self-adjoint) positive invertible operators in ${\mathcal B}(H)$. Let $A,B\in{\mathcal B}^{+*}(H)$ and $p\in[0,1]$ be a real number. The following
\begin{multline*}
A\nabla_{p}B=(1-p)A+pB,\;
A!_{p}B=\Big((1-p)A^{-1}+pB^{-1}\Big)^{-1},\\
A\sharp_{p}B=A^{1/2}\Big(A^{-1/2}BA^{-1/2}\Big)^{p}A^{1/2}
\end{multline*}
are known in the literature as the weighted arithmetic mean, weighted harmonic mean and weighted geometric mean of $A$ and $B$, respectively. If $p=1/2$ they are simply denoted by $A\nabla B,\; A!B$ and $A\sharp B$, respectively. The previous operator means satisfy the following relationships
\begin{equation}\label{eq11}
A\nabla_{p}B=B\nabla_{1-p}A,\;\; A!_{p}B=B!_{1-p}A,\;\; A\sharp_{p}B=B\sharp_{1-p}A.
\end{equation}

It is well-known that the following double inequality

\begin{equation}\label{eq12}
A!_{p}B\leq A\sharp_{p}B\leq A\nabla_{p}B
\end{equation}
holds for any $A,B\in{\mathcal B}^{+*}(H)$ and $p\in[0,1]$. Here, the notation $T\leq S$ means that $T,S\in{\mathcal B}(H)$ are self-adjoint and $S-T$ is positive semi-definite.

Otherwise, the relative operator entropy $S(A|B)$ and the Tsallis relative operator entropy $T_{p}(A|B)$ are, respectively, defined by, see \cite{FK,F1,IIKTW}
$$S(A|B)=A^{1/2}\log\big(A^{-1/2}BA^{-1/2}\big)A^{1/2},\;\;\; T_{p}(A|B)=\frac{A\sharp_{p}B-A}{p},\;\; p\neq0.$$
The following double inequality is known in the literature:
\begin{equation}\label{eq13}
A-AB^{-1}A\leq S(A|B)\leq B-A.
\end{equation}

In \cite{F2}, Furuta introduced a parametric extension of $S(A|B)$ as follows
\begin{equation}\label{eq14}
S_{p}(A|B)=A^{1/2}\big(A^{-1/2}BA^{-1/2}\big)^{p}\log\big(A^{-1/2}BA^{-1/2}\big)A^{1/2}.
\end{equation}
In fact, $S_p(A|B)$ was introduced in \cite{F2} for any real number $p$, but here we restrict ourselves to the case $p\in[0,1]$.

As pointed out in \cite{F2}, it is not hard to see that
\begin{equation*}
S_0(A|B)=S(A|B),\; S_1(A|B)=-S(B|A)\;\; \mbox{and}\;\; S_{p}(A|B)=-S_{1-p}(B|A).
\end{equation*}

The fundamental goal of this paper is to give an extension of $S_p(A|B)$ when the operator variables $A$ and $B$ are (convex) functionals. Some functional relationships and inequalities are provided as well. The related operator versions are deduced in a fast and nice way.

\section{\bf Functional Extensions}

The previous operator concepts have been extended from the case that the variables are positive operators to the case that the variables are convex functionals, see \cite{RAM1}.

Let $\tilde{\mathbb R}^{H}$ be the extended space of all functionals defined from $H$ into ${\mathbb R}\cup\{+\infty\}$. Let $f,g\in\tilde{\mathbb R}^{H}$ be two given functionals (convex or not) and $p\in(0,1)$. The following
$${\mathcal A}_{p}(f,g)=(1-p)f+pg,$$
$${\mathcal H}_{p}(f,g)=\Big((1-p)f^*+pg^*\Big)^*,$$
\begin{equation*}
{\mathcal
G}_{p}(f,g)=\displaystyle{\frac{\sin(p\pi)}{\pi}\int_{0}^1\frac{t^{p-1}}{(1-t)^{p}}}{\mathcal H}_{t}(f,g)dt
\end{equation*}
are called, by analogy, the weighted functional arithmetic mean, the weighted harmonic mean and the weighted geometric mean of $f$ and $g$, respectively. Here, the notation $f^*$ refers to the Fenchel conjugate of $f$ defined through
\begin{equation}\label{eq22}
\forall x^*\in H\;\;\;\;\;\;\; f^*(x^*)=\sup_{x\in H}\big\{\Re e\langle x^*,x\rangle-f(x)\big\}.
\end{equation}

For $p=1/2$ we will denote the previous functional means by ${\mathcal A}(f,g),\; {\mathcal H}(f,g)$ and ${\mathcal G}(f,g)$, respectively. We extend these means on the whole interval $[0,1]$ by setting:
\begin{equation}\label{eq23}
{\mathcal A}_{0}(f,g)={\mathcal H}_{0}(f,g)={\mathcal G}_{0}(f,g)=f,\;\;{\mathcal A}_{1}(f,g)={\mathcal H}_{1}(f,g)={\mathcal G}_{1}(f,g)=g.
\end{equation}
It is worth mentioning that (\ref{eq23}) is not immediate from the definition of their related functional means, since our involved functionals $f$ and/or $g$ can take the value $+\infty$, with the convention $0.(+\infty)=(+\infty)-(+\infty)=+\infty$.

With this, analogous relationships of (\ref{eq11}) for the previous functional means are also valid, i.e.
\begin{equation*}
{\mathcal A}_p(f,g)={\mathcal A}_{1-p}(g,f),\;\; {\mathcal H}_p(f,g)={\mathcal H}_{1-p}(g,f),\;\; {\mathcal G}_p(f,g)={\mathcal G}_{1-p}(g,f).
\end{equation*}
In fact, the two first relations are immediate from their definitions and for the third one see a detailed proof in \cite{RB}.  Also, analog of (\ref{eq12}), i.e.
\begin{equation}\label{eq25}
{\mathcal H}_p(f,g)\leq{\mathcal G}_p(f,g)\leq{\mathcal A}_p(f,g)
\end{equation}
holds for any $f,g\in\tilde{\mathbb R}^{H}$ and $p\in[0,1]$. Here the notation $f\leq g$ means that $g(x)-f(x)\geq0$ for all $x\in H$, with the convention $+\infty-(+\infty)=+\infty$ as usual in convex analysis. The double inequality (\ref{eq25}) implies that the three involved functional means are with finite values whenever $f$ and $g$ are as well.

In two earlier papers \cite{RAM1} and \cite{RAM2} we extended $S(A|B)$ and $T_{p}(A|B)$ from operators to (convex) functionals, respectively, as follows:
\begin{equation*}
{\mathcal S}(f|g)=\int_0^1\frac{{\mathcal H}_t(f,g)-f}{t}dt,
\end{equation*}
\begin{equation*}
{\mathcal T}_{p}(f|g)=\frac{{\mathcal G}_{p}(f,g)-f}{p},\;\; p\neq0.
\end{equation*}

The previous functional concepts were constructed as extensions of their related operator versions in the following sense: if ${\mathcal O}(A,B)$ is one of the previous operator concepts, its functional extension ${\mathcal F}(f,g)$ is such that
\begin{equation}\label{eq27}
{\mathcal F}\big(f_A,f_B\big)=f_{{\mathcal O}(A,B)},
\end{equation}
where the notation $f_T$, for any $T\in{\mathcal B(H)}$, refers to the quadratic function generated by the operator $T$, i.e.
$f_T(x)=(1/2)\langle Tx,x\rangle$ for all $x\in H$.

\section{\bf Needed Tools}

Let $f\in\tilde{\mathbb R}^{H}$. We denote by ${\rm dom}\;f:=\{x\in H,\;\; f(x)<+\infty\}$
the so-called effective domain of $f$. The notation ${\rm int}({\rm dom}\;f)$ refers to the topological interior of ${\rm dom}\;f$ in $H$. The Fenchel conjugate $f^*$ of $f$ defined by (\ref{eq22}) satisfies
$$f^*(x^*):=\sup_{x\in{\rm dom}\;f}\big\{\Re e\langle x^*,x\rangle-f(x)\big\}$$
for any $x^*\in H$. As supremum of a family of affine (so convex) functions, $f^*$ is always convex even if $f$ is not. The conjugate map $f\longmapsto f^*$ is point-wise increasing and convex. That is, $f\leq g$ implies $g^*\leq f^*$, and the inequality
$$\big((1-p)f+pg\big)^*\leq(1-p)f^*+pg^*$$
holds for any $f,g\in\tilde{\mathbb R}^{H}$ and $p\in[0,1]$.

The sub-differential of $f$ at $x\in{\rm dom}\;f$ is the set $\partial f(x)$ defined through
$$\partial f(x)=\big\{x^*\in H;\;\; \forall z\in H\;\;\;\;\; f(z)\geq f(x)+\Re e\langle x^*,z-x\rangle\big\}.$$
As it is well known, $\partial f(x)$ is a (possibly empty) convex and closed set. If $x\in{\rm int}({\rm dom}\;f)$ then $\partial f(x)\neq\emptyset$. In the case where $\partial f(x)\neq\emptyset$ then we have the equivalence:
\begin{equation*}
x^*\in \partial f(x)\Longleftrightarrow f(x)+f^*(x^*)=\Re e\langle x^*,x\rangle.
\end{equation*}

As usual we denote by $\Gamma_{0}(H)$ the cone of all functionals $f\in\tilde{\mathbb R}^{H}$ that are convex, lower semi-continuous and proper (i.e. not identically equal to $+\infty$). It is well-known that $f^{**}:=(f^*)^*\leq f$ for any $f\in\tilde{\mathbb R}^{H}$ and, $f\in\Gamma_{0}(H)$ if and only if $f=f^{**}:=(f^*)^*$. Moreover, $x^*\in\partial f(x)$ always implies $x\in\partial f^*(x^*)$, with reversed implication provided that $f\in\Gamma_{0}(H)$.

The function $f$ is called G\^{a}teaux-differentiable (in short G-differentiable) at $x$ if the directional derivative
$$f^{'}(x,d)=\lim_{t\downarrow0}\frac{f(x+td)-f(x)}{t}$$
of $f$ at $x$ exists in all direction $d\in H$ and the map $d\longmapsto f^{'}(x,d)$ is linear. In this case we write $f^{'}(x,d)=\nabla f(x).d$ and $\nabla f(x)$ is called the G-derivative of $f$ at $x$. It is well known that if $f$ is convex and G-differentiable at $x$ then $\partial f(x)=\{\nabla f(x)\}$.

For the sake of clearness and simplicity for the reader, we state the following example illustrating the previous concepts.

\begin{example}\label{ex31}
Let $A\in{\mathcal B}(H)$ and let $f_A$ be the quadratic function associated to $A$, i.e. $f_A(x)=(1/2)\langle Ax,x\rangle$ for all $x\in H$.\\
(i) Assume that $A\in{\mathcal B}^{+*}(H)$. Then $f_A$ is convex and G-differentiable on $H$ and so
$$\forall x\in H\;\;\;\;\;\; \partial f_A(x)=\{\nabla f_A(x)\}=\{Ax\}.$$
The coefficient $1/2$ appearing in $f_A$ enjoys a symmetry role, in the aim to have
$$(f_A)^*(x^*)=(1/2)\langle A^{-1}x^*,x^*\rangle\;\; \mbox{for all}\; x^*\in H,\;\; \mbox{or in short}\;\;\big(f_A\big)^*=f_{A^{-1}}.$$
(ii) For any $A,B\in{\mathcal B}(H)$ it is easy to check that $f_A\pm f_B=f_{A\pm B}$ and $f_A(Bx)=f_{BAB}(x)$ for any $x\in H$.
\end{example}

The following result, which will be needed later, has been proved in \cite{RAM3}.

\begin{theorem}\label{th31}
Let $f\in\Gamma_{\circ}(H)$ be such that ${\rm int}({\rm dom}\;f)$ is nonempty. Then\\
(i) The following inequality
\begin{equation}\label{eq32}
\sup_{x^*\in\partial f(x)}\big(f^*-g^*\big)(x^*)\leq {\mathcal S}(f/g)(x)\leq (g-f)(x)
\end{equation}
holds true for all $x\in {\rm int}({\rm dom}\;f)$.\\
(ii) If $f$ is moreover G-differentiable at $x$, then we have
\begin{equation}\label{eq33}
f^*\big(\nabla f(x)\big)-g^*\big(\nabla f(x)\big)\leq {\mathcal S}(f/g)(x)\leq (g-f)(x).
\end{equation}
\end{theorem}

As explained in \cite{RAM3}, (\ref{eq32}) as well as (\ref{eq33}) is a functional extension of (\ref{eq13}) from positive operators to convex functionals.

For the sake of simplicity for the reader, we need to introduce auxiliary notation. For $f,g\in\tilde{\mathbb R}^{H}$ and $p\in[0,1]$ we set
\begin{equation}\label{eq34}
{\mathcal T}_p^*(f|g)=\frac{\big({\mathcal G}_p(f,g)\big)^*-f^*}{p},\; p\neq0.
\end{equation}

We have the following result summarizing the elementary properties of ${\mathcal T}_p^*(f|g)$.

\begin{proposition}\label{pr31}
The following assertions hold:\\
(i) For any $p\in[0,1)$ one has
\begin{equation*}
{\mathcal T}_{1-p}^*(g|f)=\frac{\big({\mathcal G}_p(f,g)\big)^*-g^*}{1-p}.
\end{equation*}
(ii) For all $p\in(0,1]$, the left hand-side of the following inequality
\begin{equation*}
\frac{\big({\mathcal A}_p(f,g)\big)^*(x^*)-f^*(x^*)}{p}\leq{\mathcal T}_p^*(f|g)(x^*)\leq g^*(x^*)-f^*(x^*)
\end{equation*}
holds for any $x^*\in H$ while the right hand-side holds for $x^*$ such that $g^*(x^*)=+\infty$ or $x^*\in{\rm dom}\;f^*$.
\end{proposition}
\begin{proof}
(i) Follows from (\ref{eq34}) with the relation ${\mathcal G}_{p}(f,g)={\mathcal G}_{1-p}(g,f)$.\\
(ii) From (\ref{eq25}) we obtain by taking the conjugate side to side
$$\big({\mathcal A}_p(f,g)\big)^*\leq\big({\mathcal G}_p(f,g)\big)^*\leq\big({\mathcal H}_p(f,g)\big)^*.$$
Remarking that
$$\big({\mathcal H}_p(f,g)\big)^*\leq(1-p)f^*+pg^*$$
we then deduce the desired result.
\end{proof}

\begin{proposition}\label{pr32}
For any $A,B\in{\mathcal B}^{+*}(H)$ and $p\in(0,1]$ there holds
$${\mathcal T}_p^*(f_A|f_B)=f_{T_p(A^{-1}|B^{-1})}.$$
\end{proposition}
\begin{proof}
By (\ref{eq34}) with (\ref{eq27}) we have
$${\mathcal T}_p^*(f_A|f_B)=\frac{\big({\mathcal G}_p(f_A,f_B)\big)^*-f_A^*}{p}=\frac{f_{A\sharp_pB}^*-f_A^*}{p}=
\frac{f_{(A\sharp_pB)^{-1}}-f_{A^{-1}}}{p}=\frac{f_{A^{-1}\sharp_pB^{-1}}-f_{A^{-1}}}{p}.$$
This, with the fact that $\alpha f_T=f_{\alpha T}$ and $f_T-f_S=f_{T-S}$ for any $\alpha\in{\mathbb R}$ and $T,S\in{\mathcal B}(H)$, immediately yields the desired result.
\end{proof}

\section{\bf Functional Version of $S_p(A|B)$}

As already pointed before, our aim here is to give an analog of $S_p(A|B)$ when the operator arguments $A$ and $B$ are (convex) functionals $f$ and $g$, respectively. Such analog seems to be hard to define from (\ref{eq14}), since (\ref{eq14}) involves product of operators whose analogues for functionals is not known yet. For this, we need to state the following result.

\begin{theorem}
The following equalities
\begin{equation}\label{eqP}
S_p(A|B)=-\frac{S\big(A\sharp_pB|A\big)}{p}=\frac{S\big(A\sharp_pB|B\big)}{1-p}.
\end{equation}
hold for any $A,B\in{\mathcal B}^{+*}(H)$ and $p\in(0,1)$.
\end{theorem}
\begin{proof}
Indeed, we have the property
$$
T^*S(A|B)T = S(T^*AT|T^*BT)
$$
for any $A,B\in {\mathcal B}^{*}(H)$ and any invertible operator $T \in {\mathcal B}(H)$
by using Kubo-Ando theory \cite{KA} and the integral form
$$
S(A|B)=\int_0^1 \frac{A!_tB-A}{t}dt.
$$
We thus have the first equality as
\begin{multline*}
S\left(A\sharp_pB|A\right) =A^{1/2}S\left(I\sharp_pA^{-1/2}BA^{-1/2}|I\right)A^{1/2}=A^{1/2}S\left((A^{-1/2}BA^{-1/2})^p|I\right)A^{1/2}\\
=-A^{1/2}(A^{-1/2}BA^{-1/2})^p\log(A^{-1/2}BA^{-1/2})^pA^{1/2}=-pS_p(A|B),
\end{multline*}
since $S(A|I)=-A\log A$ for any $A\in{\mathcal B}^{+*}(H)$.

The second equality can be proved in a similar manner.
\end{proof}

Now, to give a functional version of $S_p(A|B)$ we use (\ref{eqP}) which is more appropriate for our aim, since (\ref{eqP}) involves only operator concepts (relative operator entropy and operator geometric mean) whose functional extensions were already done. Taking into account to have a symmetric character between $p$ and $1-p$ in our desired definition, we then put the following.

\begin{definition}
Let $f,g\in\tilde{\mathbb R}^{H}$ and $p\in[0,1]$. We set
\begin{equation}\label{eq41}
{\mathcal S}_p\big(f|g\big)=\frac{{\mathcal S}\Big(\big({\mathcal G}_p(f,g)\big)|g\Big)}{2(1-p)}
-\frac{{\mathcal S}\Big(\big({\mathcal G}_p(f,g)\big)|f\Big)}{2p},
\end{equation}
with
$${\mathcal S}_0(f|g)={\mathcal S}(f|g)\;\;\mbox{and}\;\; {\mathcal S}_1(f|g)=-{\mathcal S}(g|f).$$
\end{definition}

As first result we state the following.

\begin{proposition}\label{pr41}
Let $f,g\in\tilde{\mathbb R}^{H}$. Then we have
\begin{equation}\label{eq42}
{\mathcal S}_{1/2}(f|g)={\mathcal S}\Big(\big({\mathcal G}(f,g)\big)|g\Big)-{\mathcal S}\Big(\big({\mathcal G}(f,g)\big)|f\Big).
\end{equation}
Further, if ${\rm dom}\;f={\rm dom}\;g=H$ then the equality
\begin{equation}\label{eq43}
{\mathcal S}_p(f|g)=-{\mathcal S}_{1-p}(g|f)
\end{equation}
holds for any $p\in(0,1)$.
\end{proposition}
\begin{proof}
The equality (\ref{eq42}) is immediate from (\ref{eq41}). However, we mention that (\ref{eq43}) is not immediate from (\ref{eq41}), since our involved functionals could take the value $+\infty$. Indeed, we pay attention
to the fact that, if $\phi,\psi\in\tilde{\mathbb R}^{H}$, the equality $\phi-\psi=-(\psi-\phi)$ is not always true, unless ${\rm dom}\;\phi\cup{\rm dom}\;\psi=H$. For this reason we have assumed in our statement that ${\rm dom}\;f={\rm dom}\;g=H$ in the aim to guarantee that ${\mathcal H}_t\big({\mathcal G}_p(f,g),g\big)$ or ${\mathcal H}_t\big({\mathcal G}_p(f,g),f\big)$ is with finite values. With this, (\ref{eq43}) can be deduced from (\ref{eq41}) when we refer to the relationship ${\mathcal G}_p(\phi,\psi)={\mathcal G}_{1-p}(\psi,\phi)$ valid for any $\phi,\psi\in\tilde{\mathbb R}^{H}$ and $p\in[0,1]$.
\end{proof}

A connection relationship between the functional parametric entropy ${\mathcal S}_p(f|g)$ and the operator parametric entropy $S_p(A|B)$ is expressed by the following result.

\begin{proposition}\label{pr42}
Let $A,B\in{\mathcal B}^{+*}(H)$ and $p\in[0,1]$. Then we have:
\begin{equation}\label{eq44}
{\mathcal S}_p\big(f_A|f_B\big)=f_{S_p(A|B)}.
\end{equation}
\end{proposition}
\begin{proof}
By (\ref{eq41}), with (\ref{eq27}) and (\ref{eqP}), we have
$${\mathcal S}_p\big(f_A|f_B\big)=\frac{{\mathcal S}\big(f_{A\sharp_pB}|f_B\big)}{2(1-p)}-\frac{{\mathcal S}\big(f_{A\sharp_pB}|f_A\big)}{2p}
=\frac{f_{S\big(A\sharp_pB|B\big)}}{2(1-p)}-\frac{f_{S\big(A\sharp_pB|A\big)}}{2p}.$$
This, with similar arguments as in the proof of Proposition \ref{pr32}, implies the desired result.
\end{proof}

Relationship (\ref{eq44}) justifies that ${\mathcal S}_p(f|g)$ is a reasonable extension of $S_p(A|B)$, from operators to functionals, in the sense of (\ref{eq27}).

For the sake of simplicity, we use in the next theorem and in its proof the following notations:
$${\mathcal G}_p:={\mathcal G}_p(f,g),\; {\mathcal G}_p^*:=\big({\mathcal G}_p(f,g)\big)^*,\; \nabla{\mathcal G}_p:=\nabla\big({\mathcal G}_p(f,g)\big),\;
\partial{\mathcal G}_p:=\partial\big({\mathcal G}_p(f,g)\big).$$

We now are in a position to state the following main result.

\begin{theorem}\label{th41}
Let $f,g\in\tilde{\mathbb R}^{H}$ be such that ${\rm int}\big({\rm dom}\;{\mathcal G}_p(f,g)\big)\neq\emptyset$. Then the following double inequality
\begin{multline}\label{eq45}
\frac{1}{2}\left(\sup_{x^*\in\partial{\mathcal G}_p(x)}{\mathcal T}_{1-p}^*(g|f)(x^*)+{\mathcal T}_p(f|g)(x)\right)\leq{\mathcal S}_p(f|g)(x)\\
\leq\frac{1}{2}\left(-{\mathcal T}_{1-p}(g|f)(x)-\sup_{x^*\in\partial{\mathcal G}_p(x)}{\mathcal T}_p^*(f|g)(x^*)\right)
\end{multline}
holds for any $x\in{\rm int}\big({\rm dom}\;{\mathcal G}_p(f,g)\big)$ and $p\in(0,1)$.
\end{theorem}
\begin{proof}
Since ${\rm int}\big({\rm dom}\;{\mathcal G}_p(f,g)\big)\neq\emptyset$ then $\partial{\mathcal G}_p(x)\neq\emptyset$ for any $x\in{\rm int}\big({\rm dom}\;{\mathcal G}_p(f,g)\big)$.

Now, according to Theorem \ref{th31} we have, for $x\in{\rm int}\big({\rm dom}\;{\mathcal G}_p(f,g)\big)$
\begin{equation}\label{eq46}
\sup_{x^*\in\partial{\mathcal G}_p(x)}\big({\mathcal G}_p^*-f^*\big)(x^*)\leq{\mathcal S}\big({\mathcal G}_p|f\big)(x)\leq\big(f-{\mathcal G}_p\big)(x),
\end{equation}
and
\begin{equation}\label{eq47}
\sup_{x^*\in\partial{\mathcal G}_p(x)}\big({\mathcal G}_p^*-g^*\big)(x^*)\leq{\mathcal S}\big({\mathcal G}_p|g\big)(x)
\leq\big(g-{\mathcal G}_p\big)(x).
\end{equation}
Multiplying (\ref{eq46}) by $-1/p$ and (\ref{eq47}) by $1/(1-p)$ and then summing side to side we obtain the desired inequalities, after simple manipulations with the help of Proposition \ref{pr31}. Details are simple and therefore omitted here.
\end{proof}

\begin{remark}
It is worth mentioning that the condition ${\rm int}\big({\rm dom}\;{\mathcal G}_p(f,g)\big)\neq\emptyset$ is satisfied if ${\rm int}\big({\rm dom}\;f\cap{\rm dom}\;g\big)\neq\emptyset$, since ${\rm dom}\;f\cap{\rm dom}\;g\subset{\rm dom}\;{\mathcal G}_p(f,g)$.
\end{remark}

\begin{corollary}\label{cor41}
Let $f,g\in\tilde{\mathbb R}^{H}$ be such that ${\mathcal G}_p(f,g)$ is G-differentiable at $x\in H$. Then the following inequalities (in the point-wise order sense)
\begin{multline*}\label{eq48}
\frac{1}{2}\Big({\mathcal T}_{1-p}^*(g|f)\big(\nabla{\mathcal G}_p(f,g)\big)+{\mathcal T}_p(f|g)\Big)\leq{\mathcal S}_p(f|g)\\
\leq\frac{1}{2}\Big(-{\mathcal T}_{1-p}(g|f)-{\mathcal T}_{p}^*(f|g)\big(\nabla{\mathcal G}_p(f,g)\big)\Big)
\end{multline*}
hold for any $p\in(0,1)$.
\end{corollary}
\begin{proof}
Since ${\mathcal G}_p(f,g)$ is G-differentiable at $x$ then $\partial{\mathcal G}_p(f,g)(x)=\{\nabla{\mathcal G}_p(f,g)(x)\}$. Substituting this in (\ref{eq45}) and using the definition of the point-wise order we immediately obtain the desired inequalities.
\end{proof}

The operator version of the above theorem (and corollary) reads as follows.

\begin{corollary}\label{cor42}
Let $A,B\in{\mathcal B}^{+*}(H)$ and $p\in(0,1)$. Then we have
\begin{multline*}
\frac{1}{2}\Big((A\sharp_pB)T_{1-p}(B^{-1}|A^{-1})(A\sharp_pB)+T_p(A|B)\Big)\leq S_p(A|B)\\
\leq\frac{1}{2}\Big(-T_{1-p}(B|A)-(A\sharp_pB)T_p(A^{-1}|B^{-1})(A\sharp_pB)\Big).
\end{multline*}
\end{corollary}
\begin{proof}
Combining Corollary \ref{cor41}, Proposition \ref{pr32} and Example \ref{ex31},(ii) we obtain the desired operator inequalities after simple manipulations. Details are simple and therefore omitted here.
\end{proof}

Corollary \ref{cor42} gives the relation between Furuta parametric relative operator entropy and Tsallis relative operator entropy, in more general setting than the result in \cite[Theorem 2.3]{FM}.


\section*{Acknowledgement}
The author (S.F.) was partially supported by JSPS KAKENHI Grant Number 16K05257.

\end{document}